\documentclass[10pt]{amsart}
\usepackage{amsfonts}
\usepackage{ifthen}
\usepackage{amsthm}
\usepackage{amsmath}
\usepackage{graphicx}
\usepackage{amscd,amssymb,amsthm}
\usepackage{enumerate}
\usepackage{color}

\newcounter{minutes}
\divide\time by 60
\newcounter{hours}
\multiply\time by 60 \addtocounter{minutes}{-\time}

\title{Modified Dini functions: monotonicity patterns and functional inequalities}

\author[\'A. Baricz]{\'Arp\'ad Baricz}
\address{Department of Economics, Babe\c{s}-Bolyai University, Cluj-Napoca 400591, Romania}
\address{Institute of Applied Mathematics, John von Neumann Faculty of Informatics, \'Obuda University, 1034 Budapest, Hungary}
\email{bariczocsi@yahoo.com}

\author[S. Ponnusamy]{Saminathan Ponnusamy}
\address{Indian Statistical Institute, Chennai Centre, Society for Electronic Transactions and Security,
MGR Knowledge City, CIT Campus, Taramani, Chennai 600113, India}
\email{samy@iitm.ac.in}

\author[S. Singh]{Sanjeev Singh}
\address{Department of Mathematics, Indian Institute of Technology Madras, Chennai 600036, India}
\email{sanjeevsinghiitm@gmail.com}

\newtheorem{theorem}{Theorem}

\newtheorem{lemma}{Lemma}
\newtheorem{corollary}[theorem]{Corollary}

\begin{document}

\def\thefootnote{}
\footnotetext{ \texttt{File:~\jobname .tex,
          printed: \number\year-0\number\month-\number\day,
          \thehours.\ifnum\theminutes<10{0}\fi\theminutes}
} \makeatletter\def\thefootnote{\@arabic\c@footnote}\makeatother


\maketitle


\begin{abstract}
We deduce some new functional inequalities, like Tur\'an type inequalities,
Redheffer type inequalities, and a Mittag-Leffler expansion for a special combination of modified Bessel
functions of the first kind, called modified Dini functions. Moreover, we
show the complete monotonicity of a quotient of modified Dini functions
by introducing a new continuous infinitely divisible probability distribution.
The key tool in our proofs is a recently developed infinite product representation
for a special combination of Bessel functions of the first, which was very useful in
determining the radius of convexity of some normalized Bessel functions of the first kind.
\end{abstract}

\section{\bf Introduction and preliminaries}
\setcounter{equation}{0}

Let us start with the Dini function $d_{\nu}:\mathbb{C}\rightarrow \mathbb{C}$ defined by
$$d_{\nu}(z)=(1-\nu)J_{\nu}(z)+zJ'_{\nu}(z)=J_{\nu}(z)-zJ_{\nu+1}(z),$$
which in view of $I_{\nu}(z)=\mathrm{i}^{-\nu}J_{\nu}(\mathrm{i}z)$ gives the modified Dini function $\xi_{\nu}:\mathbb{C}\to\mathbb{C},$ defined by
$$\xi_{\nu}(z)=\mathrm{i}^{-\nu}d_{\nu}(\mathrm{i}z)=(1-\nu)I_{\nu}(z)+zI'_{\nu}(z)
=I_{\nu}(z)+zI_{\nu+1}(z).$$
In view of the Weierstrassian factorization of $d_{\nu}(z)$ (see \cite{bpogsz}),
\begin{equation}\label{product}
d_{\nu}(z)=\frac{z^{\nu}}{2^{\nu}\Gamma(\nu+1)} \prod_{n\geq 1}\left(1-\frac{z^2}{\alpha_{\nu,n}^2}\right),
\end{equation}
where $\nu>-1,$ and the formula $\xi_{\nu}(z)=\mathrm{i}^{-\nu}d_{\nu}(\mathrm{i}z)$, we have the following Weierstrassian factorization of $\xi_{\nu}(z)$ for all $\nu>-1$ and $z\in\mathbb{C}$
\begin{equation}\label{product1}
\xi_{\nu}(z)=\frac{z^{\nu}}{2^{\nu}\Gamma(\nu+1)} \prod_{n\geq 1}\left(1+\frac{z^2}{\alpha_{\nu,n}^2}\right),
\end{equation}
where the infinite product is uniformly convergent on each compact subset of the complex plane. Here $\alpha_{\nu,n}$ is $n$th
positive zero of the Dini function $d_{\nu}$. Now for $\nu>-1$ define the function $\lambda_{\nu}:\mathbb{R}\rightarrow [1,\infty)$ as
\begin{equation}\label{product2}
\lambda_{\nu}(x)=2^{\nu}\Gamma(\nu+1)x^{-\nu}\xi_{\nu}(x)=\prod_{n\geq 1}\left(1+\frac{x^2}{\alpha_{\nu,n}^2}\right).
\end{equation}

By using some ideas from \cite{baricz1,baricz3,bwu,ismail}, in this paper our aim is to deduce some new functional inequalities, like Tur\'an type inequalities,
Redheffer type inequalities for the above special combination of modified Bessel functions of the first kind, called modified Dini functions. Moreover, we
show the complete monotonicity of a quotient of modified Dini functions by introducing a new continuous infinitely divisible probability distribution.
The key tool in our proofs is the above infinite product representation, which was very useful in
determining the radius of convexity of some normalized Bessel functions \cite{basz}.

Before we present our main results, we recall some standard definitions and basic facts. We say that a function $f:[a,b]\subseteq\mathbb{R}\rightarrow \mathbb{R}$ is convex if for all $x,y\in [a,b]$ and $\alpha\in [0,1]$ we have
$$f(\alpha x+(1-\alpha)y)\leq \alpha f(x)+(1-\alpha)f(y).
$$
If the above inequality is reversed then $f$ is called a concave function. Moreover it is known that if $f$ is differentiable,
then $f$ is convex (concave) if and only if $f'$ is increasing (decreasing) and if $f$ is twice differentiable, then $f$ is
convex (concave) if and only if $f''$ is positive (negative). A function $g:[a,b]\subseteq\mathbb{R}\rightarrow (0,\infty)$ is said to be log-convex if its
natural logarithm $\log g$ is convex, that is, for all $x,y\in [a,b]$ and $\alpha\in [0,1]$ we have
$$g(\alpha x+(1-\alpha)y)\leq (g(x))^{\alpha}(g(y))^{1-\alpha}.$$
If the above inequality is reversed then $g$ is called a log-concave function. It is also known that if $g$ is differentiable,
then $g$ is log-convex (log-concave) if and only if $g'/g$ is increasing (decreasing). A function $h:[a,b]\subseteq [0,\infty)\rightarrow [0,\infty)$ is said to be geometrically convex if it is convex with respect to the geometric mean, that is, if for all $x,y\in [a,b]$ and $\alpha \in [0,1]$ we have
$$h(x^{\alpha}y^{1-\alpha})\leq (h(x))^{\alpha}(h(y))^{1-\alpha}.$$
If the above inequality is reversed then $h$ is called a geometrically concave function. If $h$ is differentiable, then $h$ is
geometrically convex (concave) if and only if $x\mapsto xh'(x)/h(x)$ is increasing (decreasing). A probability distribution $\mu$ on the half line $(0,\infty)$ is infinitely divisible if for every $n\in\{1,2,\dots\}$ there exists a probability distribution
$\mu_n$ on $(0,\infty)$ such that for each $n\in\{1,2,\dots\}$
$$\int_0^{\infty}e^{-xt}d\mu=\left(\int_0^{\infty}e^{-xt}d\mu_n\right)^{n}.$$
A function $f:(0,\infty)\rightarrow \mathbb{R}$ possessing derivatives of all order is called a completely monotonic function if
$(-1)^nf^{(n)}(x)\geq 0$ for all $x>0$ and $n\in \{0,1,\dots\}$. The classes of completely monotonic and infinitely divisible distributions
are related by the following result, see Feller \cite[p. 425]{feller}.

\begin{lemma}\label{lem2}
The function $w$ is the Laplace transform of an infinitely
divisible probability distribution if and only if $w(x)=e^{-h(x)}$, where $h$ has a completely monotone derivative and $h(0)=0$.
\end{lemma}

Finally, let us recall the following result (see \cite{biernacki,pv}) which will be used in the sequel.
\begin{lemma}\label{lem1}
Consider the power series $f(x)=\displaystyle\sum_{n\geq 0}a_{n}x^n$ and $g(x)=\displaystyle\sum_{n\geq 0}b_{n}x^n$,
where for all $n\geq 0$ we have $a_{n}\in \mathbb{R}$ and $b_{n}>0$, and suppose that both series converge on $(-r,r)$, $r>0$. If the
sequence $\lbrace a_n/b_n\rbrace_{n\geq 0}$ is increasing (decreasing), then the function $x\mapsto{f(x)}/{g(x)}$ is increasing
(decreasing) too on $(0,r)$. We note that this result remains true if we have the power series
$$f(x)=\displaystyle\sum_{n\geq 0}a_{n}x^{2n}\ \ \ \mbox{and}\ \ \ g(x)=\displaystyle\sum_{n\geq 0}b_{n}x^{2n}$$ or
$$f(x)=\displaystyle\sum_{n\geq 0}a_{n}x^{2n+1}\ \ \ \mbox{and} \ \ \ g(x)=\displaystyle\sum_{n\geq 0}b_{n}x^{2n+1}.$$
\end{lemma}

\section{\bf Monotonicity properties and inequalities for modified Dini functions}
\setcounter{equation}{0}

\subsection{Log-convexity properties and Tur\'an type inequalities} Our first set of results are some monotonicity and convexity properties of modified Dini functions as well as some Tur\'an type inequalities.

\begin{theorem}\label{th1}
The following assertions are valid:
\begin{enumerate}
\item[\bf a.] The function $x\mapsto\lambda_{\nu}(x)$ is increasing on $(0,\infty)$ for all $\nu>-1$;
\item[\bf b.] The function $x\mapsto\lambda_{\nu}(x)$ is strictly log-convex on $(-\alpha_{\nu,1},\alpha_{\nu,1})$ and
strictly geometrically convex on $(0,\infty)$ for all $\nu>-1$;
\item[\bf c.] The functions $\nu\mapsto\lambda_{\nu}(x)$ and $\nu\mapsto{x\lambda'_{\nu}(x)}/{\lambda_{\nu}(x)}$ are
decreasing on $(-1,\infty)$ for all $x\in \mathbb{R};$
\item[\bf d.] The function $\nu\mapsto\lambda_{\nu}(x)$ is log-convex on $(-1,\infty)$ for all $x\in \mathbb{R}.$
Moreover, the following reversed Tur\'an type inequality holds for all $\nu>0$ and $x\in \mathbb{R}:$
\begin{equation}\label{turan1}
\lambda_{\nu}^2(x)-\lambda_{\nu-1}(x)\lambda_{\nu+1}(x)\leq 0.
\end{equation}
In addition, the following Tur\'an type inequality holds for all $\nu>-1$ and $x>0:$
\begin{equation}\label{turan2}
-\frac{1}{\nu}\lambda^2_{\nu}(x)\leq \lambda_{\nu}^2(x)-\lambda_{\nu-1}(x)\lambda_{\nu+1}(x).
\end{equation}
\item[\bf e.] The function $\nu\mapsto {\lambda_{\nu+1}(x)}/{\lambda_{\nu}(x)}$ is increasing on $(-1,\infty)$ for all $x\in \mathbb{R};$
\item[\bf f.] The function $x\mapsto 1/\lambda_{\nu}(\sqrt{x})$ is completely monotonic on $(0,\infty)$ for all $\nu>-1$.
Moreover, the following inequality is valid for all $x,y\geq 0$ and $\nu>-1:$
\begin{equation}\label{subadditive}
\lambda_{\nu}(\sqrt{x}+\sqrt{y})\leq \lambda_{\nu}(\sqrt{x})\lambda_{\nu}(\sqrt{y}).
\end{equation}
\item[\bf g.] The function $x\mapsto\lambda_{\nu}(\sqrt{x})$ is log-concave on $(0,\infty)$ for all $\nu>-1$.
\end{enumerate}
\end{theorem}

\begin{proof}[\bf Proof]
{\bf a.} By taking the logarithmic derivative of \eqref{product2} we have
\begin{equation}\label{monotone}
(\log \lambda_{\nu}(x))'=\frac{\lambda'_{\nu}(x)}{\lambda_{\nu}(x)}=\sum_{n\geq 1}\frac{2x}{\alpha_{\nu,n}^2+x^2}.
\end{equation}
This implies that for $\nu >-1$ the function $x\mapsto\log\lambda_{\nu}(x)$ is increasing on $(0,\infty)$ and hence $x\mapsto\lambda_{\nu}(x)$ is increasing too on $(0,\infty)$ for $\nu>-1$.

{\bf b.} Differentiating \eqref{monotone} with respect to $x$ we have
$$\left( \frac{\lambda'_{\nu}(x)}{\lambda_{\nu}(x)}\right)'=\sum_{n\geq 1}\frac{2(\alpha_{\nu,n}^2-x^2)}{(\alpha_{\nu,n}^2+x^2)^2}.$$
Thus, for $\nu >-1$ the function $x\mapsto {\lambda'_{\nu}(x)}/{\lambda_{\nu}(x)}$ is strictly increasing on $(-\alpha_{\nu,1},\alpha_{\nu,1})$ and
hence the function $x\mapsto\lambda_{\nu}(x)$ is strictly log-convex on $(-\alpha_{\nu,1},\alpha_{\nu,1})$. Now, by using again \eqref{monotone} we obtain that
$$\left( \frac{x\lambda^{'}_{\nu}(x)}{\lambda_{\nu}(x)}\right)^{'}=\sum_{n\geq 1}\frac{4x\alpha_{\nu,n}^2}{(\alpha_{\nu,n}^2+x^2)^2},$$
which implies that the function $x\mapsto{x\lambda'_{\nu}(x)}/{\lambda_{\nu}(x)}$ is strictly increasing on $(0,\infty)$ for $\nu>-1$ and hence
$x\mapsto\lambda_{\nu}(x)$ is strictly geometrically convex on $(0,\infty)$ for all $\nu>-1.$

{\bf c.} In view of the infinite product representation \eqref{product2} we have,
$$\frac{\partial\log(\lambda_{\nu}(x)}{\partial\nu}=-\sum_{n\geq 1}\frac{2x^2\frac{\partial\alpha_{\nu,n}}{\partial\nu}}
{\alpha_{\nu,n}(\alpha_{\nu,n}^2+x^2)}$$
and
$$\frac{\partial}{\partial\nu}\left(\frac{x\lambda'_{\nu}(x)}
{\lambda_{\nu}(x)}\right) =-\sum_{n\geq 1}\frac{4x^2\alpha_{\nu,n}\frac{\partial\alpha_{\nu,n}}
{\partial\nu}}{(\alpha_{\nu,n}^2+x^2)^2}.$$
Now in view of \cite[p. 196]{landau}, the expression ${\partial\alpha_{\nu,n}}/{\partial\nu}$ is positive for $\nu>-1$ and hence the
functions $\nu\mapsto\lambda_{\nu}(x)$ and $\nu\mapsto{x\lambda'_{\nu}(x)}/{\lambda_{\nu}(x)}$ are decreasing on $(-1,\infty)$ for
all $x\in \mathbb{R}$.

{\bf d.} By using \eqref{product2} we have
\begin{equation}\label{md}
\lambda_{\nu}(x)=\mathcal{I}_{\nu}(x)+\frac{x^2}{2(\nu+1)}\mathcal{I}_{\nu+1}(x).
\end{equation}
Here for $\nu>-1$ the function $\mathcal{I}_{\nu}:\mathbb{R}\rightarrow [1,\infty)$ is defined by,
\begin{equation}\label{mb}
\mathcal{I}_{\nu}(x)=2^{\nu}\Gamma(\nu+1)x^{-\nu}I_{\nu}(x)=\sum_{n\geq 0}\frac{(1/4)^n}{(\nu+1)_{n}n!}x^{2n},
\end{equation}
where $(\nu+1)_{n}=(\nu+1)(\nu+2){\dots}(\nu+n)=\Gamma(\nu+n+1)/\Gamma(\nu+1).$ Using the fact that sum of log-convex functions is log-convex and that for $x\in \mathbb{R}$ the function
$\nu\mapsto\mathcal{I}_{\nu}(x)$ is log-convex on $(-1,\infty)$ (see \cite{baricz1}), to prove that $\nu\mapsto \lambda_{\nu}(x)$ is log-convex on $(-1,\infty)$ for $x\in\mathbb{R}$ it is enough to show that
$\nu\mapsto\frac{x^2}{2(\nu+1)}\mathcal{I}_{\nu+1}(x)$ is log-convex on $(-1,\infty)$ for $x\in \mathbb{R}$.
Now, the functions $\nu\mapsto\frac{x^2}{2(\nu+1)}\mathcal{I}_{\nu+1}(x)$ is log-convex if and only if
$$\nu\mapsto\log({x^2}/{2})-\log(\nu+1)+\log(\mathcal{I}_{\nu+1}(x))$$ is convex on $(-1,\infty)$.
As $\nu\mapsto -\log(\nu+1)$ and $\nu\mapsto\log(\mathcal{I}_{\nu+1}(x))$ are convex on $(-1,\infty)$ for all $x\in \mathbb{R}$, we conclude that $\nu\mapsto\log({x^2}/{2})-\log(\nu+1)+\log(\mathcal{I}_{\nu+1}(x))$ is
convex on $(-1,\infty)$ for all $x\in \mathbb{R}$ and hence $\nu\mapsto\frac{x^2}{2(\nu+1)}\mathcal{I}_{\nu+1}(x)$ is
log-convex for $\nu>-1$ and $x\in \mathbb{R}$.

Alternatively, by using the idea from \cite{baricz1} concerning the log-convexity of $\nu\mapsto\mathcal{I}_{\nu}(x)$, it can be shown that
$\nu\mapsto\frac{x^2}{2(\nu+1)}\mathcal{I}_{\nu+1}(x)$ is log-convex on $(-1,\infty)$ for $x\in \mathbb{R}$. Namely, consider the expression
$$\frac{x^2}{2(\nu+1)}\mathcal{I}_{\nu+1}(x)=\sum_{n\geq 0}\frac{(1/4)^n}{2(\nu+1)_{n+1}n!}x^{2n+2}=\sum_{n\geq 0}b_{n}(\nu)x^{2n+2},$$
where $$b_{n}(\nu)=\frac{(1/4)^n}{2(\nu+1)_{n+1}n!}.$$ To prove the log-convexity of
$\nu\mapsto\frac{x^2}{2(\nu+1)}\mathcal{I}_{\nu+1}(x)$ it is enough to show the log-convexity of
each individual terms in the above sum, that is,
$$\frac{\partial^2\log b_{n}(\nu)}{\partial\nu^2}=\psi'(\nu+1)-\psi'(\nu+n+1)\geq 0,$$
where $\psi(x)={\Gamma'(x)}/{\Gamma(x)}$ is the digamma function. But $\psi$ is concave and
hence the function $\nu\mapsto b_{n}(\nu)$ is log-convex on $(-1,\infty)$. Therefore
$\nu\mapsto\frac{x^2}{2(\nu+1)}\mathcal{I}_{\nu+1}(x)$ is log-convex on $(-1,\infty)$ for $x\in \mathbb{R}$ and consequently,
the function $\nu\mapsto\lambda_{\nu}(x)$ is log-convex on $(-1,\infty)$ for all $x\in \mathbb{R}.$

Now, to prove the inequality \eqref{turan1}, we first observe that the function $\nu\mapsto\lambda_{\nu}(x)$ is
log-convex on $(-1,\infty)$ for all $x\in \mathbb{R}$ and hence for all $\nu_{1},\nu_{2}>-1,$ $x\in \mathbb{R}$ and
$\alpha \in [0,1]$ we have
$$\lambda_{\alpha \nu_{1}+(1-\alpha)\nu_{2}}(x)\leq \left(\lambda_{\nu_{1}}(x)\right)^{\alpha}\left(\lambda_{\nu_{2}}(x)\right)^{1-\alpha}.$$
Taking $\nu_{1}=\nu-1,$ $\nu_{2}=\nu+1$ and $\alpha=\frac{1}{2}$ we get the Tur\'an type inequality \eqref{turan1} for $\nu>0$ and $x\in\mathbb{R}.$

To prove the inequality \eqref{turan2} let us consider the Tur\'anian
$$\Delta_{\nu}(x)=\xi^2_{\nu}(x)-\xi_{\nu-1}(x)\xi_{\nu+1}(x),$$
which can be rewritten as
$$\Delta_{\nu}(x)=\left[I_{\nu}^2(x)-I_{\nu-1}(x)I_{\nu+1}(x)\right]+x^2\left[I_{\nu+1}^2(x)-I_{\nu}(x)I_{\nu+2}(x)\right]
+x\left[I_{\nu}(x)I_{\nu+1}(x)-I_{\nu-1}(x)I_{\nu+2}(x)\right].$$
Using the Tur\'an inequality for modified Bessel function (see \cite{baricz2}),
\begin{equation}\label{turan4}I_{\nu}^2(x)-I_{\nu-1}(x)I_{\nu+1}(x)\geq 0,\end{equation}
which holds for $\nu>-1$ and $x\in \mathbb{R},$ and by changing the parameter $\nu$ to $\nu+1$ in it, we get
$$I_{\nu+1}^2(x)-I_{\nu}(x)I_{\nu+2}(x)\geq 0.$$
We also note that (see \cite{bpog}) for $\nu>-1$ and $x>0$ we have
\begin{equation}\label{turan5}I_{\nu}(x)I_{\nu+1}(x)-I_{\nu-1}(x)I_{\nu+2}(x)\geq 0,\end{equation}
and therefore $\Delta_{\nu}(x)\geq 0$ for $\nu>-1$ and $x>0.$ This completes the proof of Tur\'an type inequality \eqref{turan2}.

{\bf e.} From part {\bf d} the function $\nu \mapsto \log\lambda_{\nu}(x)$ is convex and hence
$\nu \mapsto \log[\lambda_{\nu+\epsilon}(x)]-\log[\lambda_{\nu}(x)]$ is increasing for all $\epsilon>0.$ In particular, by
taking $\epsilon=1$ we get that the function $\nu\mapsto {\lambda_{\nu+1}(x)}/{\lambda_{\nu}(x)}$ is increasing on $(-1,\infty)$ for all $x\in\mathbb{R}$.

{\bf f.}  The infinite product representation \eqref{product2} implies that
$$\left( -\log\frac{1}{\lambda_{\nu}(\sqrt{x})}\right) '=\sum_{n\geq 1}\frac{1}{\alpha^2_{\nu,n}+x},$$
which is completely monotonic on $(0,\infty)$ for each fixed $\nu>-1$ as an infinite series of completely monotonic functions. Since $h:(0,\infty)\rightarrow(0,\infty)$ is completely monotonic
whenever $(-\log h)'$ is completely monotonic (see \cite{alzer}), we conclude that the function
$x\mapsto 1/\lambda_{\nu}(\sqrt{x})$ is completely monotonic on $(0,\infty)$ for all $\nu>-1$. The result of
Kimberling \cite{kimberling} asserts that if $f$ is a completely monotonic function from $[0,\infty)$ into (0,1] then
$$f(x+y)\geq f(x)f(y)$$ for all $x,y \geq 0.$
Applying this result to the function $x\mapsto 1/\lambda_{\nu}(\sqrt{x})$, the inequality \eqref{subadditive} follows.

{\bf g.} From part {\bf f} of this theorem it follows that
$x\mapsto 1/\lambda_{\nu}(\sqrt{x})$ is log-convex on $(0,\infty)$ for all $\nu>-1$, since every completely monotonic
function is log convex, see \cite[p. 167]{widder}. Consequently, $x\mapsto \lambda_{\nu}(\sqrt{x})$ is log-concave on $(0,\infty)$
for all $\nu>-1.$ Note that another proof of the log-concavity of $x\mapsto \lambda_{\nu}(\sqrt{x})$ can be given by using the infinite product
representation \eqref{product2}. Namely, from \eqref{product2} we have
$$\log(\lambda_{\nu}(\sqrt{x}))=\sum_{n\geq 1}\log\left(1+\frac{x}{\alpha^2_{\nu,n}}\right).$$
Since $x\mapsto \log\left(1+{x}/{\alpha^2_{\nu,n}}\right)$ is concave on $(0,\infty)$ for all $\nu>-1$ and for
all $n\geq 1$ it follows that $x\mapsto \log(\lambda_{\nu}(\sqrt{x}))$ is concave as an infinite sum of concave functions. Hence
$x\mapsto \lambda_{\nu}(\sqrt{x})$ is log-concave on $(0,\infty)$ for all $\nu>-1.$
\end{proof}

\subsection{Monotonicity of some quotients} Now, we are going to prove some other monotonicity properties of the modified Dini functions by using Lemma \ref{lem2}. Moreover, we present some simple bounds
for these functions in terms of hyperbolic functions. The idea of this result comes from the survey paper \cite{baricz3}, where a similar result has been proved for modified Bessel functions of the first kind.

\begin{theorem}\label{th2}
Let $\mu,\nu>-1$ and $k$ be a non-negative integer. Then the following assertions are valid:
\begin{itemize}
\item[\bf a.] If $\mu>\nu$ $(\mu<\nu)$, then the function $x\mapsto{\lambda_{\nu}(x)}/{\lambda_{\mu}(x)}$ is
strictly increasing (decreasing) on $(0,\infty)$;
\item[\bf b.] If $-1< \nu <\frac{1}{2}$ $\left(\nu > \frac{1}{2}\right)$, then $x\mapsto{\lambda_{\nu}^{(2k)}(x)}/{\cosh x}$
is strictly increasing (decreasing) on $(0,\infty)$;
\item[\bf c.] If $-1< \nu <\frac{1}{2}$ $\left(\nu > \frac{1}{2}\right)$, then $x\mapsto{\lambda_{\nu}^{(2k+1)}(x)}/{\sinh x}$
is strictly increasing (decreasing) on $(0,\infty)$;
\item[\bf d.] If $-1< \nu <\frac{1}{2}$ $\left(\nu > \frac{1}{2}\right)$, then $x\mapsto{\lambda_{\nu}(x)}/{\cosh x}$
is strictly increasing (decreasing) on $(0,\infty)$;
\item[\bf e.] If $-1< \nu < -\frac{1}{2}$ $\left(\nu > -\frac{1}{2}\right)$, then $x\mapsto{\lambda_{\nu}(x)}/{(\cosh x + x\sinh x)}$
is strictly increasing (decreasing) on $(0,\infty)$;
\item[\bf f.] The following inequalities are valid for all $x>0:$
\begin{equation}\label{ineq1}
\lambda_{\nu}(x) > \cosh x, ~~\mbox{for}~~ \nu\in \left(-1,\frac{1}{2}\right),
\end{equation}
\begin{equation}\label{ineq2}
\lambda_{\nu}(x) < \cosh x, ~~\mbox{for}~~ \nu > \frac{1}{2},
\end{equation}
\begin{equation}\label{ineq3}
\lambda_{\nu}(x) > \cosh x + x\sinh x, ~~\mbox{for}~~  \nu\in \left(-1,-\frac{1}{2}\right),
\end{equation}
\hspace{-1.3cm} and
\begin{equation}\label{ineq4}
\lambda_{\nu}(x) < \cosh x + x\sinh x, ~~\mbox{for}~~   \nu > -\frac{1}{2}.
\end{equation}
Moreover all the above inequalities are sharp when $x=0$.
\end{itemize}
\end{theorem}

\begin{proof}[\bf Proof]
{\bf a.} Using \eqref{md} and \eqref{mb} we have the following power series
\begin{equation}\label{series1}
\lambda_{\nu}(x)=\sum_{n\geq 0}\frac{(2n+1)x^{2n}}{4^nn!(\nu+1)_n}.
\end{equation}
In view of Lemma \ref{lem1} and the power series representations of $\lambda_{\nu}(x)$ and $\lambda_{\mu}(x)$, it is enough
to study the monotonicity of the sequence $\lbrace \alpha_n\rbrace_{n\geq 0}={(\mu+1)_n}/{(\nu+1)_n}$. Now, it can be checked
that for all $n\in\{0,1,\dots\}$ we have ${\alpha_{n+1}}/{\alpha_n}=(\mu+n+1)/(\nu+n+1)> 1$ if and only if $\mu >\nu$,
and hence the conclusion follows.

{\bf b.} By using \eqref{series1} we obtain that
$$\lambda^{(2k)}_{\nu}(x)=\sum_{n\geq 0}\frac{(2n+2k+1)!}{(2n)!4^{n+k}(n+k)!(\nu+1)_{n+k}}x^{2n}$$
and
$$\lambda^{(2k+1)}_{\nu}(x)=\sum_{n\geq 0}\frac{(2n+2k+3)!}{(2n+1)!4^{n+k+1}(n+k+1)!(\nu+1)_{n+k+1}}x^{2n+1}.$$
We also note that $$\mathcal{I}_{-1/2}(x)=\cosh x,\ \ \mathcal{I}_{1/2}(x)=\frac{\sinh x}{x}\ \ ~~\mbox{and}~~ \mathcal{I}_{3/2}(x)=-3\left(\frac{\sinh x}{x^3}-\frac{\cosh x}{x^2}\right).$$ Hence using \eqref{md} we have,
\begin{equation}\label{cosh}
\lambda_{1/2}(x)=\cosh x=\sum_{n\geq0}\frac{x^{2n}}{(2n)!},
\end{equation}
\begin{equation}\label{cosh+xsinh}
\lambda_{-1/2}(x)=\cosh x + x\sinh x = \sum_{n\geq0}\frac{(2n+1)x^{2n}}{4^nn!(1/2)_n},
\end{equation}
and
\begin{equation}\label{sinh}
\frac{\lambda_{-1/2}(x)-\lambda_{1/2}(x)}{x}= \sinh x = \sum_{n\geq 0}\frac{x^{2n+1}}{(2n+1)!}.
\end{equation}
By using Lemma \ref{lem1} and \eqref{cosh}, it is enough to verify the monotonicity of the sequence $\{\alpha_{n}\}_{n\geq 0}$
where
$$\alpha_{n}=\frac{(2n+2k+1)!}{4^{n+k}(n+k)!(\nu+1)_{n+k}}.$$
But, ${\alpha_{n+1}}/{\alpha_{n}}={(2n+2k+3)}/{(2\nu+2n+2k+2)}>1$ if and only if $\nu>\frac{1}{2}$ and the conclusion follows.

{\bf c.} Again using Lemma \ref{lem1} and \eqref{sinh} the result follows as the sequence $\{\beta_{n}\}_{n\geq 0}$ where
$$\beta_{n}=\frac{(2n+2k+3)!}{4^{n+k+1}(n+k+1)!(\nu+1)_{n+k+1}},
$$
strictly increases for $-1<\nu<\frac{1}{2}$ and decreases  for $\nu>\frac{1}{2}$.

{\bf d.} This follows from part {\bf a} by taking $\mu=\frac{1}{2}$ and observing that $\lambda_{1/2}(x)=\cosh x$.
Alternatively, this can be proved from part {\bf b} by choosing $k=0$.

{\bf e.} This part again follows from part {\bf a} by taking $\mu=-\frac{1}{2}$ and noticing $\lambda_{-1/2}(x)=\cosh x + x\sinh x$.

{\bf f.} The inequalities \eqref{ineq1} and \eqref{ineq2} follow from part {\bf d} while \eqref{ineq3} and \eqref{ineq4} follow from part {\bf e}.

We note that for $-1<\nu<-1/2$ the inequality \eqref{ineq3} improves the inequality \eqref{ineq1} and for $\nu>1/2$
the inequality \eqref{ineq2} improves the inequality \eqref{ineq4}.
\end{proof}

\subsection{\bf Concluding remarks and further results} Now, we would like to comment on the previous results. We first note that the Tur\'an type inequality \eqref{turan4} and the fact that $I_{\nu}(x)>0$ for all $\nu>-1$ and $x>0$ actually imply the Tur\'an type inequality \eqref{turan5} for $\nu>0$ and $x>0.$ Moreover, by using the following power series representation of the product of modified Bessel functions \cite[p. 252]{nist}
$$I_{\nu}(x)I_{\mu}(x)=\sum _{k\geq 0}\frac{(\nu+\mu+k+1)_k(\frac{x}{2})^{2k+\nu+\mu}}{k!\Gamma(\nu+k+1)\Gamma(\mu+k+1)}$$
we can also prove the inequality \eqref{turan5} for $\nu>-1$ and $x>0.$

We also note that by using the power series representation \eqref{series1} and \cite[Theorem 3]{karp}, we can get another proof for
the log-convexity of the function $\nu\mapsto\lambda_{\nu}(x)$ on $(-1,\infty)$ for all $x\in \mathbb{R}$. Moreover, if we consider the expression $f(\mu,x)={\lambda_{\nu}(x)}/{\Gamma(\nu+1)},$ where $\mu=\nu+1>0$ and $x\in \mathbb{R}$, then by using \cite[Theorem 3.1]{kalmykov} we can conclude that the function $\nu\mapsto {\lambda_{\nu}(x)}/{\Gamma(\nu+1)}$ is log-concave on $(-1,\infty)$ for each fixed $x\in \mathbb{R}$
which in turn implies the Tur\'an type inequality \eqref{turan2} for $\nu>0$. Now, using the equations (17) and (19) from \cite{kalmykov}, we have the following bounds for the Tur\'anian of
${\lambda_{\nu}(x)}/{\Gamma(\nu+1)}$ for $\nu>0$ and $x\in \mathbb{R}:$
$$\frac{1}{(\nu+1)\Gamma^2(\nu+1)} < \frac{\lambda^2_{\nu}(x)}{\Gamma^2(\nu+1)}-\frac{\lambda_{\nu-1}(x)\lambda_{\nu+1}(x)}{\Gamma(\nu)\Gamma(\nu+2)}
 < \frac{1}{\nu+1}\frac{\lambda^2_{\nu}(x)}{\Gamma^2(\nu+1)}.$$
We note that the right-hand side of the above Tur\'an type inequality is equivalent to \eqref{turan1} for $\nu>0,$ while the left-hand side gives the following Tur\'an type inequality
$$\frac{1}{\nu}-\frac{1}{\nu}\lambda^2_{\nu}(x)\leq \lambda_{\nu}^2(x)-\lambda_{\nu-1}(x)\lambda_{\nu+1}(x).$$
where $\nu>0$ and $x\in \mathbb{R}.$ This improves the Tur\'an type inequality \eqref{turan2}.

We mention that part {\bf a} of Theorem \ref{th2} can also be proved using part {\bf c} of Theorem \ref{th1} for all $x>0.$
Namely, for $\mu>\nu>-1$ and $x>0$ we have
$$\frac{\lambda'_{\nu}(x)}{\lambda_{\nu}(x)}
>\frac{\lambda'_{\mu}(x)}{\lambda_{\mu}(x)},$$
which implies that
$$\left( \frac{\lambda_{\nu}(x)}{\lambda_{\mu}(x)}\right)' =\frac{\lambda_{\mu}(x)\lambda'_{\nu}(x)-\lambda'_{\mu}(x)\lambda_{\nu}(x)}
{\lambda^2_{\mu}(x)}>0.$$

Finally, we would like to mention that by using the Weierstrassian decomposition \eqref{product2} it is
possible to deduce a Mittag-Leffler type expansion for the function $\lambda_{\nu}.$ Namely, by using the infinite product representation \eqref{product2} we have the following Mittag-Leffler expansion for all $\nu>-1$ and $x\in \mathbb{R}:$
\begin{equation}\label{mittag1}
\frac{\lambda_{\nu+1}(x)}{\lambda_{\nu}(x)}=-\frac{4(\nu+1)}{x^2-1+2\nu} + \frac{4(\nu+1)(x^2+1+2\nu)}{x^2-1+2\nu}\sum_{n\geq 1}\frac{1}{\alpha_{\nu,n}^2+x^2}.
\end{equation}
To prove the above expression, we note that
\begin{align*}
\lambda_{\nu}(x)&=2^{\nu}\Gamma(\nu+1)x^{-\nu}[I_{\nu}(x)+xI_{\nu+1}(x)]\\
\lambda'_{\nu}(x)&=2^{\nu}\Gamma(\nu+1)x^{-\nu}[xI_{\nu}(x)+(1-2\nu)I_{\nu+1}(x)],
~~\mbox{and}\\
\lambda_{\nu+1}(x)&=2^{\nu+1}\Gamma(\nu+2)x^{-\nu-1}[xI_{\nu}(x)-(2\nu+1)I_{\nu+1}(x)]
\end{align*}
Here we have used the formula \cite[p. 252]{nist} $$(x^{-\nu}I_{\nu}(x))'=x^{-\nu}I_{\nu+1}(x)$$ and the recurrence relations
\cite[p. 251]{nist} $$xI'_{\nu}(x)+\nu I_{\nu}(x)=xI_{\nu-1}(x) ~~\mbox{and}~~xI_{\nu-1}(x)-xI_{\nu+1}=2\nu I_{\nu+1}(x).$$ Combining the above equations on $\lambda_{\nu}(x),$ $\lambda'_{\nu}(x)$ and $\lambda_{\nu+1}(x)$ we obtain that
$$\lambda_{\nu+1}(x)=\frac{2(\nu+1)}{x(x^2-1+2\nu)}\left[ -2x\lambda_{\nu}(x)+(x^2+1+2\nu)\lambda'_{\nu}(x)\right],$$
which in view of \eqref{monotone} gives \eqref{mittag1}.

\subsection{Monotonicity properties of the Dini functions} Our third set of main results are some monotonicity properties for the Dini function itself, which is a special combination of Bessel functions of the first kind. For $\nu>-1$, let us define the function $\mathcal{D}_{\nu}:\mathbb{R}\rightarrow\mathbb{R}$ by
\begin{equation}\label{product3}
\mathcal{D}_{\nu}(x)=2^{\nu}\Gamma(\nu+1)x^{-\nu}d_{\nu}(x)=\prod_{n\geq 1}\left(1-\frac{x^2}{\alpha_{\nu,n}^2}\right)
\end{equation}
where $\alpha_{\nu,n}$ is the $n$th positive zero of the Dini function $d_{\nu}$. Now by the definition of $\mathcal{D}_{\nu}(x)$ we have
\begin{equation}\label{d1}
\mathcal{D}_{\nu}(x)=\mathcal{J}_{\nu}(x)-\frac{x^2}{2(\nu+1)}\mathcal{J}_{\nu+1}(x),
\end{equation}
where $\nu>-1$ and the function $\mathcal{J}_{\nu}:\mathbb{R}\rightarrow (-\infty,1]$ is defined by
$$\mathcal{J}_{\nu}(x)=2^{\nu}\Gamma(\nu+1)x^{-\nu}J_{\nu}(x)=\sum_{n\geq 0}\frac{(-1/4)^n}{(\nu+1)_{n}n!}x^{2n}.$$
By using \eqref{d1} we have the following power series for $\mathcal{D}_{\nu}(x)$,
$$\mathcal{D}_{\nu}(x)=\sum_{n\geq 0}\frac{(-1)^n(2n+1)x^{2n}}{4^nn!(\nu+1)_n}.$$
Using the above power series we get the following expression for the derivative of $\mathcal{D}_{\nu}(x)$
\begin{equation}\label{derivative3}
\mathcal{D}'_{\nu}(x)=-\frac{x}{2(\nu+1)}\mathcal{D}_{\nu+1}(x)-\frac{x}{\nu+1}\mathcal{J}_{\nu+1}(x).
\end{equation}
We also note that using the power series \eqref{series1}, we get the following expression for the derivative of $\lambda_{\nu}(x)$
\begin{equation}\label{derivative4}
\lambda'_{\nu}(x)=\frac{x}{2(\nu+1)}\lambda_{\nu+1}(x)+\frac{x}{\nu+1}\mathcal{I}_{\nu+1}(x).
\end{equation}

The next results may be proved by using some ideas from \cite[Theorem 3]{baricz1}.

\begin{theorem}\label{th3}
Let $\nu>-1$, and define $\triangle=\triangle_1\cup\triangle_2$, where
$\triangle_1=\bigcup_{n\geq 1}[-\alpha_{\nu,2n},-\alpha_{\nu,2n-1}]$ and $\triangle_2=\bigcup_{n\geq 1}[\alpha_{\nu,2n-1},\alpha_{\nu,2n}]$.
Then the following are valid
\begin{enumerate}
\item[\bf a.] The function $x\mapsto \mathcal{D}_{\nu}(x)$ is negative on $\triangle$ and strictly positive on $\mathbb{R}\setminus\triangle$;
\item[\bf b.] The function $x\mapsto \mathcal{D}_{\nu}(x)$ is increasing on $(-\alpha_{\nu,1},0]$ and decreasing on $[0,\alpha_{\nu,1})$;
\item[\bf c.] The function $x\mapsto \mathcal{D}_{\nu}(x)$ is strictly log-concave on $\mathbb{R}\setminus\triangle$;
\item[\bf d.] The function $x\mapsto d_{\nu}(x)$ is strictly log-concave on $(0,\infty)\setminus\triangle_2$, provided $\nu>0$;
\item[\bf e.] The function $\nu \mapsto \mathcal{D}_{\nu}(x)$ is increasing on $(-1,\infty)$ for all $x\in (-\alpha_{\nu,1},\alpha_{\nu,1}).$
\end{enumerate}
\end{theorem}

\begin{proof}[\bf Proof]
{\bf a.} From the infinite product representation \eqref{product3} and the fact that $$0<\alpha_{\nu,1}<\alpha_{\nu,2}<{\cdots}<\alpha_{\nu,n}<{\cdots},$$
we see that if $x\in [\alpha_{\nu,2n-1},\alpha_{\nu,2n}]$ or $x\in [-\alpha_{\nu,2n},-\alpha_{\nu,2n-1}]$ then the first $(2n-1)$
terms of the product \eqref{product3} are negative and the remaining terms are strictly positive. Therefore $\mathcal{D}_{\nu}(x)$
becomes negative on $\triangle$. Now if $x\in (-\alpha_{\nu,1},\alpha_{\nu,1})$ then each terms of the product \eqref{product3}
are strictly positive and if $x\in (\alpha_{\nu,2n},\alpha_{\nu,2n+1})$ or $x\in (-\alpha_{\nu,2n+1},-\alpha_{\nu,2n})$, then
the first $2n$ terms are strictly negative while the rest is strictly positive. Therefore the function $\mathcal{D}_{\nu}(x)>0$
on $\mathbb{R}\setminus\triangle$.

{\bf b.} From part {\bf a} we have $\mathcal{D}_{\nu}(x)>0$ on $(-\alpha_{\nu,1},\alpha_{\nu,1})$. Therefore by taking
the logarithmic derivative of both sides of \eqref{product3}, we have
$$(\log \mathcal{D}_{\nu}(x))'=\frac{\mathcal{D}'_{\nu}(x)}{\mathcal{D}_{\nu}(x)}=-\sum_{n\geq 1}\frac{2x}{\alpha_{\nu,n}^2-x^2}.$$
From this we conclude that the function $x\mapsto \mathcal{D}_{\nu}(x)$ is increasing on $(-\alpha_{\nu,1},0]$ and decreasing on $[0,\alpha_{\nu,1})$.

{\bf c.} Differentiating both sides of the above relation with respect to $x$, we have
$$(\log \mathcal{D}_{\nu}(x))''=-\sum_{n\geq 1}\frac{2(\alpha_{\nu,n}^2+x^2)}{(\alpha_{\nu,n}^2-x^2)^2}.$$
Thus we conclude that the function $x\mapsto \mathcal{D}_{\nu}(x)$ is strictly log-concave on $\mathbb{R}\setminus\triangle$.
Note that this part has been proved also in \cite[Theorem 4]{bpogsz} but only for $x\in (0,\infty)\setminus\triangle_2$.

{\bf d.} From \eqref{product3} we have
$$d_{\nu}(x)=\frac{x^{\nu}\mathcal{D}_{\nu}(x)}{2^{\nu}\Gamma(\nu+1)}.$$
Now from part {\bf c} and using the fact that product of log-concave function is log-concave, the conclusion follows, as $x\mapsto x^{\nu}$
is log-concave on $(0,\infty)$ for all $\nu\geq 0$. Another proof can be seen in \cite[Theorem 3]{bpogsz}.

{\bf e.} We note that this part has been proved in \cite[Theorem 6]{bpogsz} for $\nu\in (-1,\infty)$ and $x\in (0,\alpha_{\nu,1})$
but because of the following expression
$$\frac{\partial}{\partial\nu}\left( \log\left(\mathcal{D}_{\nu}(x)\right)\right)
=\sum_{n\geq 1}\frac{2x^2\frac{\partial\alpha_{\nu,n}}{\partial\nu}}{\alpha_{\nu,n}(\alpha_{\nu,n}^2-x^2)},$$
the result is true for $\nu\in (-1,\infty)$ and $x\in (-\alpha_{\nu,1},\alpha_{\nu,1})$.
\end{proof}

Now, we show another result on Dini functions and modified Dini functions.

\begin{theorem}\label{th4}
Let $\nu>-1$. Then the function $x\mapsto {\lambda_{\nu}(x)}/{\mathcal{D}_{\nu}(x)}={\xi_{\nu}(x)}/{d_{\nu}(x)}$
is strictly log-convex on $(-\alpha_{\nu,1},\alpha_{\nu,1})$. Moreover, the following inequality holds for all $x,y\in(-\alpha_{\nu,1},\alpha_{\nu,1}),$
$$\frac{\lambda^2_{\nu}(\frac{x+y}{2})}{\mathcal{D}^2_{\nu}(\frac{x+y}{2})}
\leq \frac{\lambda_{\nu}(x)\lambda_{\nu}(y)}{\mathcal{D}_{\nu}(x)\mathcal{D}_{\nu}(y)},
$$
and in particular, for all $x,y\in(-\alpha_{-1/2,1},\alpha_{1/2,1})$ we have
$$\frac{\left(\cosh\left(\frac{x+y}{2}\right)+\left(\frac{x+y}{2}\right)\sinh\left(\frac{x+y}{2}\right)\right)^2}{\left(\cosh x+x\sinh x\right)\left(\cosh y+y\sinh y\right)}
\leq \frac{\left(\cos\left(\frac{x+y}{2}\right)-\left(\frac{x+y}{2}\right)\sin\left(\frac{x+y}{2}\right)\right)^2}{\left(\cos x-x\sin x\right)\left(\cos y-y\sin y\right)},$$
where $\alpha_{-1/2,1}\simeq0.8603335890{\dots}$ is the first positive root of the equation $\cos x=x\sin x.$
\end{theorem}

\begin{proof}[\bf Proof]
By using part {\bf b} of Theorem \ref{th1}, the function $x\mapsto\lambda_{\nu}(x)$ is strictly log-convex on $(-\alpha_{\nu,1},\alpha_{\nu,1})$
and by part {\bf c} of Theorem \ref{th3}  the function $x\mapsto 1/\mathcal{D}_{\nu}(x)$ is strictly log-convex on
$(-\alpha_{\nu,1},\alpha_{\nu,1})$. Therefore, the function $x\mapsto {\lambda_{\nu}(x)}/{\mathcal{D}_{\nu}(x)}={\xi_{\nu}(x)}/{d_{\nu}(x)}$
is strictly log-convex on $(-\alpha_{\nu,1},\alpha_{\nu,1})$, as product of two strictly log-convex functions. The first inequality
follows by definition of log-convexity and the other inequality follows from \eqref{cosh+xsinh} and  observing the
following special value of $\mathcal{D}_{\nu}(x)$
$$\mathcal{D}_{-1/2}(x)=\cos x - x\sin x,$$
which can be derived by using $$\mathcal{J}_{-1/2}(x)=\cos x,\ \ \mathcal{J}_{1/2}(x)=\frac{\sin x}{x}.$$
\end{proof}

\subsection{An infinitely divisible probability distribution involving Dini functions} The next result is motivated by \cite[Theorem 1.9]{ismail}. The next distributions are very natural companions to the distributions considered by Ismail and Kelker \cite{ismail}.

\begin{theorem}\label{th5}
Let $\mu >\nu>-1$. Then the function $x\mapsto \lambda_{\mu}(\sqrt{x})/\lambda_{\nu}(\sqrt{x})$ is a completely monotonic function with
$$\frac{\lambda_{\mu}(\sqrt{x})}{\lambda_{\nu}(\sqrt{x})}=\int_0^{\infty}e^{-xt}\rho(t,\nu,\mu)dt,$$
where $\rho(t,\nu,\mu)$ is a probability density function, on $(0,\infty)$, of an infinitely divisible distribution.
\end{theorem}

An immediate consequence of Theorem \ref{th5} (taking $\mu\rightarrow\infty$) is the following corollary.

\begin{corollary}\label{cor1}
Let $\nu>-1$. Then we have
$$\frac{1}{\lambda_{\nu}(\sqrt{x})}=\int_0^{\infty}e^{-xt}\rho(t,\nu,\infty)dt,$$
where $\rho(t,\nu,\infty)$ is an infinitely divisible
probability density.
\end{corollary}

We remark that part {\bf f} of Theorem \ref{th1} may be obtained as a consequence of this corollary.

\begin{proof}[\bf Proof of Theorem \ref{th5}]
Let us consider $h(x)=- \log \left({\lambda_{\mu}(\sqrt{x})}/{\lambda_{\nu}(\sqrt{x})}\right)$ and hence using \eqref{product2} we have
\begin{align*}
h'(x)&= \left( \log \lambda_{\nu}(\sqrt{x})\right) '- \left( \log \lambda_{\mu}(\sqrt{x})\right)'=\sum_{n\geq 1}\frac{1}{\alpha_{\nu,n}^2+x}-\sum_{n\geq 1}\frac{1}{\alpha_{\mu,n}^2+x}=\sum_{n\geq 1}\frac{\alpha_{\mu,n}^2-\alpha_{\nu,n}^2}{(\alpha_{\nu,n}^2+x)(\alpha_{\mu,n}^2+x)}.
\end{align*}
Since $\alpha_{\mu,n}>\alpha_{\nu,n} $ for $\mu>\nu$ and for each $n\in \{1,2,\ldots\}$ \cite[p. 196]{landau},
therefore each term in above series is positive and completely monotonic which implies that $x\mapsto h'(x)$ is completely monotonic
as a sum of completely monotonic functions, consequently in view of \cite[Lemma 2.4]{alzer} the function
$x\mapsto \lambda_{\mu}(\sqrt{x})/\lambda_{\nu}(\sqrt{x})$ is a completely monotonic function on $(0,\infty)$, as we required.
Now as $x\mapsto h'(x)$ is completely monotonic and from \eqref{product2}, $h(0)=0$ and hence by Lemma \ref{lem2},
$x\mapsto\lambda_{\mu}(\sqrt{x})/\lambda_{\nu}(\sqrt{x})$ is the Laplace transform of an infinitely divisible probability distribution.
\end{proof}

\subsection{\bf Redheffer-type inequalities for modified Dini functions} In this subsection we prove some Redheffer-type inequalities for modified Dini functions. Similar investigations were carried out in \cite{bwu} for Bessel functions and modified Bessel functions. Here, we will also study the monotonicity of the product of Dini function and modified Dini function.

\begin{theorem}\label{th9}
Let $\nu>-1$ and $x\in (-\alpha_{\nu,1},\alpha_{\nu,1})$. Then the modified Dini function satisfies following sharp
exponential Redheffer-type inequality
\begin{equation}\label{redheffer}
\left( \frac{\alpha^2_{\nu,1}+x^2}{\alpha^2_{\nu,1}-x^2}\right)^{a_{\nu}}\leq \lambda_{\nu}(x)
\leq \left( \frac{\alpha^2_{\nu,1}+x^2}{\alpha^2_{\nu,1}-x^2}\right)^{b_{\nu}},
\end{equation}
where $a_{\nu}=0$ and $b_{\nu}=\frac{3\alpha^2_{\nu,1}}{8(\nu+1)}$ are the best possible constants. In particular, the following exponential Redheffer type inequality is also valid
$$\left( \frac{\alpha^2_{-1/2,1}+x^2}{\alpha^2_{-1/2,1}-x^2}\right)^{a_{-1/2}}\leq \cosh(x)+x\sinh(x)
\leq \left(\frac{\alpha^2_{-1/2,1}+x^2}{\alpha^2_{-1/2,1}-x^2}\right)^{b_{-1/2}},$$
where $a_{-1/2}=0$ and $b_{-1/2}=\frac{3}{4}\alpha^2_{-1/2,1}\simeq0.5551304132{\dots}$ are the best possible constants, and $\alpha_{-1/2,1}\simeq0.8603335890{\dots}$ is the first positive root of the equation $\cos x=x\sin x.$
\end{theorem}

\begin{proof}[\bf Proof]
Since all the functions in \eqref{redheffer} are even in $x$ it is enough to prove the result for
$x\in (0,\alpha_{\nu,1})$. From part {\bf a} of Theorem \ref{th1} the function $x\mapsto\lambda_{\nu}(x)$ is
increasing on $(0,\alpha_{\nu,1})$ for all $\nu>-1$ and hence $\lambda_{\nu}(x)\geq 1$ which gives the left-hand side of
\eqref{redheffer}. To prove the right-hand side of \eqref{redheffer}, we consider the function $f_{\nu}:(0,\alpha_{\nu,1})\mapsto \mathbb{R}$, defined by
$$f_{\nu}(x)=\frac{3\alpha^2_{\nu,1}}{8(\nu+1)}\log \left( \frac{\alpha^2_{\nu,1}+x^2}{\alpha^2_{\nu,1}-x^2}\right)-\log \lambda_{\nu}(x),$$
which in view of \eqref{monotone} and the formula \cite{bpogsz},
$$\sum_{n\geq 1}\frac{1}{\alpha^2_{\nu,n}}=\frac{3}{4(\nu+1)}$$
give
\begin{align*}
f'_{\nu}(x)&=\frac{3\alpha^2_{\nu,1}}{8(\nu+1)}.\frac{4x\alpha^2_{\nu,1}}{\alpha^4_{\nu,1}-x^4}-\sum_{n\geq 1}\frac{2x}{\alpha_{\nu,n}^2+x^2}\\
&=\frac{2x\alpha^4_{\nu,1}}{\alpha^4_{\nu,1}-x^4}\sum_{n\geq 1}\frac{1}{\alpha^2_{\nu,n}}-\sum_{n\geq 1}\frac{2x}{\alpha_{\nu,n}^2+x^2}\\
&=2x^3\sum_{n\geq 1}\frac{\alpha^4_{\nu,1}+\alpha^2_{\nu,n}x^2}{\alpha^2_{\nu,n}(\alpha^4_{\nu,1}-x^4)(\alpha^2_{\nu,n}+x^2)}.
\end{align*}
Therefore the function $f_{\nu}$ is increasing on $[0,\alpha_{\nu,1})$ for all $\nu>-1$ and hence
$f_{\nu}(x)\geq f_{\nu}(0)=0$ which implies the right-hand side of \eqref{redheffer}. Now, to prove that $a_{\nu}=0$ and
$b_{\nu}=\frac{3\alpha^2_{\nu,1}}{8(\nu+1)}$ are the best possible constants, we consider the function
$g_{\nu}:(0,\alpha_{\nu,1})\mapsto \mathbb{R}$ defined as
$$g_{\nu}(x)=\frac{\log \lambda_{\nu}(x)}{\log \left( \frac{\alpha^2_{\nu,1}+x^2}{\alpha^2_{\nu,1}-x^2}\right)}.$$
We note that $\displaystyle\lim_{x\rightarrow \alpha_{\nu,1}} g_{\nu}(x)=0=a_{\nu}$ and using the l'Hospital rule we have,
$$\lim_{x\rightarrow 0}g_{\nu}(x)=\lim_{x\rightarrow 0}\frac{\lambda'_{\nu}(x)}{\lambda_{\nu}(x)}.\frac{\alpha^4_{\nu,1}-x^4}{4x\alpha^2_{\nu,1}}
=\lim_{x\rightarrow 0}\sum_{n\geq 1}\frac{2x}{\alpha_{\nu,n}^2+x^2}.\frac{\alpha^4_{\nu,1}-x^4}{4x\alpha^2_{\nu,1}}=b_{\nu}.$$
Therefore $a_{\nu}=0$ and $b_{\nu}=\frac{3\alpha^2_{\nu,1}}{8(\nu+1)}$ are indeed the best possible constants. Alternatively,
inequality \eqref{redheffer} can be proved using the monotone form of l'Hospital's rule \cite[Lemma 2.2]{anderson}.
Namely, it is enough to observe that
$$x\mapsto \frac{\frac{d}{dx}\log \lambda_{\nu}(x)}{\frac{d}{dx}\log \left( \frac{\alpha^2_{\nu,1}+x^2}{\alpha^2_{\nu,1}-x^2}\right)}
=\frac{1}{2\alpha^2_{\nu,1}}\sum_{n\geq 1}\frac{\alpha^4_{\nu,1}-x^4}{\alpha_{\nu,n}^2+x^2}
$$
is decreasing on $(0,\alpha_{\nu,1})$ as each terms in the above series are decreasing. Therefore $g_{\nu}$ is
decreasing too on $(0,\alpha_{\nu,1})$ and hence $$a_{\nu}=\lim_{x\to\alpha_{\nu,1}}g_{\nu}(x) < g_{\nu}(x) < \lim_{x\to 0}g_{\nu}(x)=b_{\nu},$$
which gives the inequality \eqref{redheffer}.
\end{proof}

We continue with another result on Dini and modified Dini functions.

\begin{theorem}\label{th10}
Let $\nu>-1.$ The following assertions are valid:
\begin{enumerate}
\item[\bf a.] The function $x\mapsto \mathcal{D}_{\nu}(x)\lambda_{\nu}(x)$ is increasing on $(-\alpha_{\nu,1},0]$ and
decreasing on $[0,\alpha_{\nu,1})$;
\item[\bf b.] The function $\nu \mapsto \mathcal{D}_{\nu}(x)\lambda_{\nu}(x)$ is increasing on $(-1,\infty)$ for all
$x\in (-\alpha_{\nu,1},\alpha_{\nu,1})$ fixed;
\item[\bf c.] The following inequalities hold:
$$0< \mathcal{D}_{\nu}(x)\lambda_{\nu}(x)< \mathcal{D}_{\nu+1}(x)\lambda_{\nu+1}(x) < 1,$$
for all $x\in (-\alpha_{\nu,1},\alpha_{\nu,1})$ and $\nu>-1$.
\end{enumerate}
\end{theorem}

\begin{proof}[\bf Proof]
{\bf a.} Using \eqref{product2} and \eqref{product3} we get,
\begin{align*}
\left[\log(\mathcal{D}_{\nu}(x)\lambda_{\nu}(x))\right]'
=\frac{\mathcal{D}'_{\nu}(x)}{\mathcal{D}_{\nu}(x)}+\frac{\lambda'_{\nu}(x)}{\lambda_{\nu}(x)}
=-\sum_{n\geq 1}\frac{2x}{\alpha_{\nu,n}^2-x^2}+\sum_{n\geq 1}\frac{2x}{\alpha_{\nu,n}^2+x^2}
=-\sum_{n\geq 1}\frac{4x^3}{(\alpha_{\nu,n}^2-x^2)(\alpha_{\nu,n}^2+x^2)}
\end{align*}
where $\nu>-1$ and $x\in (-\alpha_{\nu,1},\alpha_{\nu,1})$. Therefore from above expression, the conclusion follows.

Alternatively, this part can be proved as follows. Using \eqref{derivative3} and \eqref{derivative4} we have
\begin{eqnarray*}
\left(\mathcal{D}_{\nu}(x)\lambda_{\nu}(x)\right)'
&=&\mathcal{D}_{\nu}(x)\lambda'_{\nu}(x)+\mathcal{D}'_{\nu}(x)\lambda_{\nu}(x)\\
&=&\mathcal{D}_{\nu}(x)\left[ \frac{x}{2(\nu+1)}\lambda_{\nu+1}(x)
+\frac{x}{\nu+1}\mathcal{I}_{\nu+1}(x)\right] \\
&&+ \left[ -\frac{x}{2(\nu+1)}\mathcal{D}_{\nu+1}(x)-
\frac{x}{\nu+1}\mathcal{J}_{\nu+1}(x)\right]\lambda_{\nu}(x)\\
&=&\frac{x}{2(\nu+1)}\left[\mathcal{D}_{\nu}(x)\lambda_{\nu+1}(x)
-\mathcal{D}_{\nu+1}(x)\lambda_{\nu}(x)\right]
+\frac{x}{\nu+1}\left[\mathcal{D}_{\nu}(x)\mathcal{I}_{\nu+1}(x)
-\mathcal{J}_{\nu+1}(x)\lambda_{\nu}(x)\right]
\end{eqnarray*}
From the above expression it is enough to prove that $$\mathcal{D}_{\nu}(x)\lambda_{\nu+1}(x)-\mathcal{D}_{\nu+1}(x)\lambda_{\nu}(x)<0
~~\mbox{and}~~ \mathcal{D}_{\nu}(x)\mathcal{I}_{\nu+1}(x)-\mathcal{J}_{\nu+1}(x)\lambda_{\nu}(x)<0$$ for all $\nu>-1$ and
$x\in [0,\alpha_{\nu,1})$. Now as the function $\nu \mapsto \mathcal{D}_{\nu}(x)$ is increasing on $(-1,\infty)$
for each fixed $x\in (-\alpha_{\nu,1},\alpha_{\nu,1})$ and the function $\nu\mapsto\lambda_{\nu}(x)$ is decreasing on $(-1,\infty)$
for all $x\in \mathbb{R}$ fixed, we have
$$\frac{\lambda_{\nu+1}(x)}{\lambda_{\nu}(x)}\leq 1\leq \frac{\mathcal{D}_{\nu+1}(x)}{\mathcal{D}_{\nu}(x)},
$$
and hence $\mathcal{D}_{\nu}(x)\lambda_{\nu+1}(x)-\mathcal{D}_{\nu+1}(x)\lambda_{\nu}(x)<0$ for all $\nu>-1$ and
$x\in [0,\alpha_{\nu,1})$. Alternatively, in view of the infinite product representation \eqref{product2}
and \eqref{product3}, it is enough to show the following inequality for all $\nu>-1$, $n\in \{1,2,\ldots\}$ and $x\in (-\alpha_{\nu,1},\alpha_{\nu,1}):$
$$\left(1-\frac{x^2}{\alpha_{\nu,n}^2}\right)
\left(1+\frac{x^2}{\alpha_{\nu+1,n}^2}\right)\leq \left(1-\frac{x^2}{\alpha_{\nu+1,n}^2}\right)
\left(1+\frac{x^2}{\alpha_{\nu,n}^2}\right),
$$
which is indeed true. Here we used the fact \cite[p. 196]{landau}, $\alpha_{\nu,n}<\alpha_{\nu+1,n}$ holds for all $\nu>-1$
and $n\in \{1,2,\ldots\}$. Now to prove the second inequality, recall the infinite product representation of
$\mathcal{J}_{\nu}(z)$ and $\mathcal{I}_{\nu}(z)$ \cite{watson}, namely
$$\mathcal{J}_{\nu}(z)=\prod_{n\geq 1}\left(1-\frac{z^2}{j_{\nu,n}^2}\right),\ \ \ \mathcal{I}_{\nu}(z)=\prod_{n\geq 1}\left(1+\frac{z^2}{j_{\nu,n}^2}\right),$$
where $j_{\nu,n}$ is the $n$th positive zero of the Bessel function $J_{\nu}$. In view of the above infinite product representations \eqref{product2} and \eqref{product3}, it is enough to show the following inequality for all $\nu>-1$, $n\in \{1,2,\ldots\}$ and $x\in (-\alpha_{\nu,1},\alpha_{\nu,1}):$
$$\left(1-\frac{x^2}{\alpha_{\nu,n}^2}\right)\left(1+\frac{x^2}{j_{\nu+1,n}^2}\right)
\leq \left(1-\frac{x^2}{j_{\nu+1,n}^2}\right)\left(1+\frac{x^2}{\alpha_{\nu,n}^2}\right);
$$
that is, $\alpha^2_{\nu,n}<j^2_{\nu+1,n}$, which is indeed true because of the following inequality
\begin{equation}\label{interlace1}
\alpha_{\nu,n}<\alpha_{\nu+1,n}<j_{\nu+1,n}.
\end{equation}
The first inequality in \eqref{interlace1} follows from the monotonicity of $\nu\mapsto \alpha_{\nu,n}$ \cite[p. 196]{landau},
and second inequality follows from Dixon's theorem \cite[p.480]{watson}, which says that when $\nu>-1$ and $a,b,c,d$ are
constants such that $ad\neq bc$, then the positive zeros of $x\mapsto aJ_{\nu}(x)+bxJ'_{\nu}(x)$ are interlaced with those of
$x\mapsto cJ_{\nu}(x)+dxJ'_{\nu}(x)$. Therefore if we choose $a=1-\nu$, $b=c=1$ and $d=0$ then for $\nu>-1$ we have,
$j_{\nu-1,n}<\alpha_{\nu,n}<j_{\nu,n},$ $n\geq 2,$
and for $n=1$, $0<\alpha_{\nu,1}<j_{\nu,1}$.

{\bf b.} Since $\nu\mapsto \alpha_{\nu,n}$ is increasing on $(-1,\infty)$ for each $n\in \{1,2,\ldots\}$, it follows that the function $\nu\mapsto\log(1-x^4/\alpha_{\nu,n})$ is increasing on $(-1,\infty)$ for each $n\in \{1,2,\ldots\}$
and $x\in (-\alpha_{\nu,1},\alpha_{\nu,1})$ fixed. Again using the infinite products \eqref{product2} and \eqref{product3},
the function
$$\nu\mapsto\log[\mathcal{D}_{\nu}(x)\lambda_{\nu}(x)]=\sum_{n\geq 1}\log\left(1-\frac{x^4}{\alpha_{\nu,n}^4}\right)$$
is  increasing on $(-1,\infty)$ for each $x\in (-\alpha_{\nu,1},\alpha_{\nu,1})$ fixed and hence the conclusion follows.

{\bf c.} This is an immediate consequence of parts {\bf a} and {\bf b} of this theorem.
\end{proof}

\subsection{\bf Bounds for Bessel and modified Bessel functions} It is important to mention here that by using a similar approach as in \cite[Remark C]{bps} we can find bounds for Dini and modified Dini functions in terms of Bessel and modified Bessel functions, which in turn give bounds for ratios of Bessel and modified Bessel functions. Namely, by Dixon's theorem \cite[p. 480]{watson} for all $n\geq 2$ and $\nu>-1$ we have $j_{\nu,n-1}<\alpha_{\nu,n}<j_{\nu,n}$, where $j_{\nu,n}$ is the $n$th positive zero of Bessel function $J_{\nu}$. Therefore by these inequalities for $\nu>-1$ and $x\in\mathbb{R}$ we have,
$$\prod_{n\geq 2}\left(1+\frac{x^2}{j_{\nu,n}^2}\right) < \prod_{n\geq 2}\left(1+\frac{x^2}{\alpha_{\nu,n}^2}\right)
< \prod_{n\geq 2}\left(1+\frac{x^2}{j_{\nu,n-1}^2}\right),
$$
which in view of \eqref{product1} and the infinite product representation of the modified Bessel function $I_{\nu}$ implies
$$\frac{j^2_{\nu,1}}{\alpha^2_{\nu,1}}.\frac{\alpha^2_{\nu,1}+x^2}{j^2_{\nu,1}+x^2}I_{\nu}(x) < \xi_{\nu}(x)
< \frac{\alpha^2_{\nu,1}+x^2}{\alpha^2_{\nu,1}}I_{\nu}(x).$$
Now, using the definition of $\xi_{\nu}(x)$ and the fact that $\alpha_{\nu,1} < j_{\nu,1}$ (see \cite{ismail1}) the above inequlaity gives the
following inequalities for all $\nu>-1$ and $x>0$
$$0 < \frac{x}{\alpha^2_{\nu,1}}.\frac{j^2_{\nu,1}-\alpha^2_{\nu,1}}{j^2_{\nu,1}+x^2} < \frac{I_{\nu+1}(x)}{I_{\nu}(x)}
< \frac{x}{\alpha^2_{\nu,1}}.$$
Using the formula
$$\mathcal{I}'_{\nu}(x)=2^{\nu}\Gamma(\nu+1)(x^{-\nu}I_{\nu}(x))'
=2^{\nu}\Gamma(\nu+1)x^{-\nu}I_{\nu+1}(x),$$
the above inequality is equivalent to
\begin{equation}\label{bound3}
\frac{t}{\alpha^2_{\nu,1}}.\frac{j^2_{\nu,1}-\alpha^2_{\nu,1}}{j^2_{\nu,1}+t^2} < \frac{\mathcal{I}'_{\nu}(t)}{\mathcal{I}_{\nu}(t)}
< \frac{t}{\alpha^2_{\nu,1}},
\end{equation}
where $t>0$ and $\nu>-1.$ Integrating \eqref{bound3} we obtain
$$\int_0^x\frac{j^2_{\nu,1}-\alpha^2_{\nu,1}}{2\alpha^2_{\nu,1}}(\log (j^2_{\nu,1}+t^2))'dt
< \int_0^x\left( \log \mathcal{I}_{\nu}(t)\right)'dt < \int_0^x\left( \frac{t^2}{2\alpha^2_{\nu,1}}\right)'dt,$$
which implies that
\begin{equation}\label{bound4}
\left(1+\frac{x^2}{j^2_{\nu,1}}\right)^{\frac{j^2_{\nu,1}-\alpha^2_{\nu,1}}{2\alpha^2_{\nu,1}}} < \mathcal{I}_{\nu}(x)
< e^{\frac{x^2}{2\alpha^2_{\nu,1}}}.
\end{equation}
Here we have used the fact that $\mathcal{I}_{\nu}(0)=1$. We also note that when $x=0$, the above
inequality \eqref{bound4} is sharp. The left-hand side of the inequality \eqref{bound4} is stronger than the
inequality $\mathcal{I}_{\nu}(x)>1$ for all $x>0$ and $\nu>-1/2,$ given by Luke \cite{luke}, while using the inequality (see \cite{ismail1}) $\alpha^2_{\nu,1} < 2(\nu+1)$, we conclude that the right-hand side of the inequality \eqref{bound4} is
weaker than the existing inequality $$I_{\nu}(x) < \frac{x^{\nu}}{2^{\nu}\Gamma(\nu+1)}e^{\frac{x^2}{4(\nu+1)}}$$ for
all $x>0$ and $\nu>-1,$ given in \cite{baricz3}.

Again integrating \eqref{bound3} over $0<x<y$, we have
$$\int_x^y\frac{j^2_{\nu,1}-\alpha^2_{\nu,1}}{2\alpha^2_{\nu,1}}(\log (j^2_{\nu,1}+t^2))'dt
< \int_x^y\left( \log \mathcal{I}_{\nu}(t)\right)'dt < \int_x^y\left( \frac{t^2}{2\alpha^2_{\nu,1}}\right)'dt
$$
and
\begin{equation}\label{bound5}
\left(\frac{x}{y}\right)^{\nu}e^{\frac{x^2-y^2}{2\alpha^2_{\nu,1}}} < \frac{I_{\nu}(x)}{I_{\nu}(y)}
< \left(\frac{x}{y}\right)^{\nu}\left(\frac{j^2_{\nu,1}+x^2}{j^2_{\nu,1}+y^2}\right)^{\frac{j^2_{\nu,1}-\alpha^2_{\nu,1}}{2\alpha^2_{\nu,1}}}.
\end{equation}
In view of the inequality (see \cite{ismail1}) $\alpha^2_{\nu,1} < 2(\nu+1)$ the left-hand side of \eqref{bound5} improves the
following inequality given by Joshi and Bissu \cite{joshi}
$$\left(\frac{x}{y}\right)^{\nu}e^{\frac{x^2-y^2}{4(\nu +1)}} < \frac{I_{\nu}(x)}{I_{\nu}(y)},$$
where $\nu>-1$ and $0<x<y.$

Next we find a bound for Dini functions in terms of Bessel functions, which in turns gives a bound for the ratio
${J_{\nu+1}(x)}/{J_{\nu}(x)}$. For $\nu>-1$, using the inequalities $j_{\nu,n-1}<\alpha_{\nu,n}<j_{\nu,n}$,
$n\geq 2$ and $0<\alpha_{\nu,1}<j_{\nu,1}$, we have for all $x\in (-\alpha_{\nu,1},\alpha_{\nu,1})$,
$$\prod_{n\geq 2}\left(1-\frac{x^2}{j_{\nu,n-1}^2}\right) < \prod_{n\geq 2}\left(1-\frac{x^2}{\alpha_{\nu,n}^2}\right)
< \prod_{n\geq 2}\left(1-\frac{x^2}{j_{\nu,n}^2}\right),
$$
which in view of \eqref{product} and the infinite product representation of the Bessel functions of the first kind implies
\begin{equation}\label{bound6}
\frac{\alpha^2_{\nu,1}-x^2}{\alpha^2_{\nu,1}}J_{\nu}(x) < d_{\nu}(x)
< \frac{j^2_{\nu,1}}{\alpha^2_{\nu,1}}.\frac{\alpha^2_{\nu,1}-x^2}{j^2_{\nu,1}-x^2}J_{\nu}(x)
\end{equation}
Now, using the definition of $d_{\nu}(x)$ and the fact that $0<\alpha_{\nu,1} < j_{\nu,1}$, \eqref{bound6} gives the
following inequality for all $\nu>-1$ and $x\in (0,\alpha_{\nu,1})$
\begin{equation}\label{bound7}
0 < \frac{x}{\alpha^2_{\nu,1}}.\frac{j^2_{\nu,1}-\alpha^2_{\nu,1}}{j^2_{\nu,1}-x^2} < \frac{J_{\nu+1}(x)}{J_{\nu}(x)}
< \frac{x}{\alpha^2_{\nu,1}}.
\end{equation}
In view of the formula $$\mathcal{J}'_{\nu}(x)=2^{\nu}\Gamma(\nu+1)(x^{-\nu}J_{\nu}(x))' =-2^{\nu}\Gamma(\nu+1)x^{-\nu}J_{\nu+1}(x),$$
the above inequality \eqref{bound7} implies
\begin{equation}\label{bound8}
\frac{t}{\alpha^2_{\nu,1}}.\frac{j^2_{\nu,1}-\alpha^2_{\nu,1}}{j^2_{\nu,1}-t^2}
< -\frac{\mathcal{J}'_{\nu}(t)}{\mathcal{J}_{\nu}(t)} < \frac{t}{\alpha^2_{\nu,1}}.
\end{equation}
Integrating \eqref{bound8} as above, we have for all $\nu>-1$ and $0<x<y<\alpha_{\nu,1}$
$$e^{-\frac{x^2}{2\alpha^2_{\nu,1}}} < \mathcal{J}_{\nu}(x)
< \left(1-\frac{x^2}{j^2_{\nu,1}}\right)^{\frac{j^2_{\nu,1}-\alpha^2_{\nu,1}}{2\alpha^2_{\nu,1}}}.$$
and
$$\left(\frac{x}{y}\right)^{\nu}\left(\frac{j^2_{\nu,1}-x^2}{j^2_{\nu,1}-y^2}\right)^{\frac{j^2_{\nu,1}-\alpha^2_{\nu,1}}{2\alpha^2_{\nu,1}}}
< \frac{J_{\nu}(x)}{J_{\nu}(y)} < \left(\frac{x}{y}\right)^{\nu}e^{-\frac{x^2-y^2}{2\alpha^2_{\nu,1}}}.$$

\subsection*{Acknowledgments} The work of \'A. Baricz was supported by the J\'anos Bolyai Research Scholarship of
the Hungarian Academy of Sciences. The second author is on leave from the Department of Mathematics,
Indian Institute of Technology Madras, Chennai-600 036, India. The research of S. Singh was supported by the
fellowship of the University Grants Commission, India.

\end{document}